\documentclass[12pt]{article}
\usepackage{hyperref}
\usepackage{cmap}                        
\usepackage[cp1251]{inputenc}            
\usepackage[english]{babel}
\usepackage[left=1in,right=1in,top=1in,bottom=1.5in]{geometry} 
\usepackage{amssymb,amsmath, amsthm, amscd,ifthen}
\usepackage{graphicx}

\newcommand{\bb}{{\mathcal B}}

\newtheorem{thm}{Theorem}
\newtheorem*{thm*}{Theorem}
\newtheorem*{cor*}{Corollary}

\newtheorem{lem}[thm]{Lemma}

\date{}

\newtheorem{prop}[thm]{Proposition}

\DeclareMathOperator{\aaa}{\mathcal A}
\DeclareMathOperator{\col}{\underset{c}\prec}

\title{A size-sensitive inequality for cross-intersecting families}
\author{Peter Frankl\footnote{R\'enyi Institute, Budapest;  Email: {\tt peter.frankl@gmail.com}}, Andrey Kupavskii\footnote{Moscow Institute of Physics and Technology, G-SCOP, CNRS; Email: {\tt kupavskii@yandex.ru} \ \ Research supported in part by the Swiss National Science Foundation Grants 200021-137574 and 200020-14453 and by the grant N 15-01-03530 of the Russian Foundation for Basic Research.}}

\date{}

\begin{document}
\maketitle
\begin{abstract} Two families $\mathcal A$ and $\mathcal B$ of $k$-subsets of an $n$-set are called cross-intersecting if $A\cap B\ne\emptyset$ for all $A\in \mathcal A, B\in \bb $. Strengthening the classical Erd\H os-Ko-Rado theorem, Pyber proved that $|\aaa||\bb|\le {n-1\choose k-1}^2$ holds for $n\ge 2k$. In the present paper we sharpen this inequality. We prove that assuming $|\bb|\ge {n-1\choose k-1}+{n-i\choose k-i+1}$ for some $3\le i\le k+1$ the stronger inequality $$|\aaa||\bb|\le \Bigl({n-1\choose k-1}+{n-i\choose k-i+1}\Bigr)\Bigl({n-1\choose k-1}-{n-i\choose k-1}\Bigr)$$ holds. These inequalities are best possible. We also present a new short proof of Pyber's inequality and a short computation-free proof of an inequality due to Frankl and Tokushige.
\end{abstract}

\section{Introduction}
Let $[n] = \{1,\ldots, n\}$ and ${[n]\choose k}$ be the family of all $k$-subsets of $[n]$ for $n\ge k\ge 0$. For a family $\mathcal F\subset {[n]\choose k}$ let $\mathcal F^c$ be the family of complements, i.e., $\mathcal F^c = \{[n]-F:F\in \mathcal F\}$. Obviously, $\mathcal F^c\subset {[n]\choose n-k}$ holds.

Two families $\aaa,\bb\subset {[n]\choose k}$ are said to be \textit{cross-intersecting} if $A\cap B\ne \emptyset$ holds for all $A\in \aaa$, $B \in \bb$. To avoid trivialities we assume $n\ge 2k.$ Analogously, $\mathcal C, \mathcal D\subset {[n]\choose l}$ are called \textit{cross-union} if $C\cup D\ne[n]$ holds for all $C\in \mathcal C, D\in \mathcal D$. Here we assume $n\le 2l$ in general. Note that $\aaa,\bb $ are cross-intersecting iff $\mathcal A^c,\bb^c$ are cross-union.

In order to state one of the most fundamental theorems in \textit{extremal set theory}, let us say that $\mathcal F\subset {[n]\choose k}$ is \textit{intersecting} if $F\cap F'\ne \emptyset $ for all $F,F' \in \mathcal F.$

\begin{thm*}[Erd\H os-Ko-Rado Theorem \cite{EKR}] If $\mathcal F\subset {[n]\choose k} $ is intersecting and $n\ge 2k >0$ then

\begin{equation}\label{eq1} |\mathcal F|\le {n-1\choose k-1} \ \ \ \ \text{holds.}
\end{equation}
\end{thm*}

Hilton and Milner \cite{HM} proved in a stronger form that for $n>2k$ the only way to achieve equality in (\ref{eq1}) is to take all $k$-subsets containing some fixed element of $[n]$.

If $\mathcal F$ is intersecting then $\aaa = \mathcal F$, $\bb = \mathcal F$ are cross-intersecting. Therefore the following result is a strengthening of (\ref{eq1}).

\begin{thm*}[Pyber's inequality \cite{P}] Suppose that $\aaa, \bb\subset {[n]\choose k}$ are cross-intersecting and $n\ge 2k$. Then
\begin{equation}\label{eq2} |\mathcal A||\bb|\le {n-1\choose k-1}^2 \ \ \ \ \text{holds.}
\end{equation}
\end{thm*}
Let us mention that the notion of cross-intersection is not just a natural extension of the notion of intersecting for two families, but it is also a very useful tool for proving results for \textit{one} family. As a matter of fact, it was already used in the paper of Hilton and Milner \cite{HM}. This explains the interest in two-family versions of intersection theorems (cf. e.g. \cite{B}, \cite{MT}, \cite{RV}).

The object of this paper is two-fold. First we provide a very short proof of (\ref{eq2}). Then we use the ideas of this proof and some counting based on the Kruskal-Katona Theorem \cite{Kr}, \cite{Ka} to obtain the following sharper, best possible bounds.\\

\textbf{Example 1.} Let $i$ be an integer and define
$\bb_i = \{B\in{[n]\choose k}: 1\in B\}\cup\{B\in{[n]\choose k}:1\notin B, [2,i]\subset B\}$, $\aaa_i = \{A\in{[n]\choose k}:1\in A, [2,i]\cap A\ne \emptyset\}.$
Note that $\aaa_i,\bb_i$ are cross intersecting with
$$|\aaa_i| = {n-1\choose k-1}-{n-i\choose k-1}, \ \ \ \ \ \ \ \ \ \ |\bb_i| = {n-1\choose k-1}+{n-i\choose k-i+1}.$$
The inequalities (\ref{eq3}) and (\ref{eq4}) given below show that the pair $(\aaa_i,\bb_i)$ is extremal in the corresponding range.\\

\begin{thm}\label{thm1} Let $\aaa,\bb\subset{[n]\choose k}$ be cross-intersecting, $n> 2k>0$ and suppose $|\aaa|\le {n-1\choose k-1}\le |\bb|$ and $\cap_{B\in \bb}B=\emptyset$. Then
\begin{equation}\label{eq3} |\aaa||\bb|\le \Bigl({n-1\choose k-1}+1\Bigr)\Bigl({n-1\choose k-1}-{n-k-1\choose k-1}\Bigr) \ \ \ \ \ \text{holds.}
\end{equation}
\end{thm}

\begin{thm}\label{thm2}
 Let $\mathcal A,\mathcal B\subset {[n]\choose k}$ be cross-intersecting, $n\ge 2k>0$ and suppose that $|\mathcal B|\ge {n-1\choose k-1}+{n-i\choose k-i+1}$ holds for some $3\le i\le k+1$. Then
\begin{equation}\label{eq4} |\mathcal A||\mathcal B|\le \Bigl({n-1\choose k-1}+{n-i\choose k-i+1}\Bigr)\Bigl({n-1\choose k-1}-{n-i\choose k-1}\Bigr).\end{equation}
\end{thm}

Note that plugging in $i = k+1$ into (\ref{eq4}) gives (\ref{eq3}) except for the case $|\bb| ={n-1\choose k-1}$ which we treat separately. In the proof of Theorem \ref{thm2} we give a short computation-free proof of an important special case of an inequality due to Frankl and Tokushige \cite{FT}, which is based on the K\"onig-Hall theorem \cite{Kon}, \cite{Hal}.

\section{Preliminaries}
Let us define the \textit{lexicographic} and the \textit{colex} orders on the $k$-subsets of $[n]$. We have $F\prec G$ in the lexicographic order if $\min F\setminus G<\min G\setminus F$ holds. E.g., $\{1,100\}\prec\{2,3\}$. The colex order $\col$ is defined by $F\col G$ if $\max F\setminus G<\max G\setminus F$. Thus, $\{2,3\}\col\{1,100\}$. Note that $\{1,3\}$ precedes $\{2,4\}$ in both orderings.

For $0\le m\le {n\choose k}$ let $\mathcal L^{(k)}(m)$ denote the \textit{initial segment} of $k$-sets of length $m$, i.e., the first $m$ $k$-sets in the lexicographic order. Note that $\mathcal L^{(k)}\bigl({n-1\choose k-1}\bigr) = \{F\in{[n]\choose k}: 1\in F\}.$

Similarly, $\mathcal C^{(k)}(m)$ denotes the family of first $m$ $k$-sets in the colex order. Note that  for $k\le s\le n$ one has $\mathcal C^{(k)}\bigl({s\choose k}\bigr) = {[s]\choose k}.$

Let us state the Kruskal-Katona  Theorem \cite{Kr}, \cite{Ka}, which is one of the most important results in extremal set theory. For $0\le t\le k$ define the $t$-shadow $\mathcal S^{(t)}(\mathcal F)$ of a family $\mathcal F\subset {[n]\choose k}$ by
$$\mathcal S^{(t)}(\mathcal F) = \Bigl\{S\in {[n]\choose t}: \exists F\in\mathcal F, S\subset F\Bigr\}.$$

\begin{thm*}[Kruskal-Katona \cite{Kr}, \cite{Ka}] The inequality
\begin{equation}\label{eq5} |\mathcal S^{(t)}(\mathcal F)|\ge |\mathcal S^{(t)}(\mathcal C^{(k)}(|\mathcal F|))|
\end{equation}
holds for all $\mathcal F\subset {[n]\choose k}$, $n\ge k\ge t\ge 0$.
\end{thm*}

Depending on the value of $|\mathcal F|$, equality might hold in (\ref{eq5}) even if $\mathcal F$ is not isomorphic to $\mathcal C^{(k)}(|\mathcal F|)$. However, F\"uredi and Griggs \cite{FG} succeeded in determining which are the values of $|\mathcal F|$ such that $\mathcal C^{(k)}(|\mathcal F|)$ is the \textit{unique} optimal family. M\"ors \cite{M} proved a stronger inequality under the assumption that $\cup_{F\in\mathcal F} F = [n]$. We need the following special case of it.

\begin{thm*}[M\"ors, \cite{M}] Suppose that $\mathcal G\subset {[n]\choose l}$, $n\ge l>t\ge1$ and $|\mathcal G| = {n-1\choose l}$. If $\cup_{G\in\mathcal G} = [n]$ holds then
\begin{equation}\label{eq6} |\mathcal S^{(t)}(\mathcal G)|\ge {n-1\choose t}+{l-1\choose t-1}.
\end{equation}
\end{thm*}

Computationwise, the bounds arising from the Kruskal-Katona Theorem are not easy to handle. Lov\'asz \cite{L} found the following slightly weaker but very handy form.
\begin{thm*}[Lov\'asz, \cite{L}] If $n\ge k\ge t\ge 0$, $\mathcal F\subset{[n]\choose k}$ and $|\mathcal F| = {x\choose k}$ for a real number $x\ge k$ then
\begin{equation}\label{eq7} |\mathcal S^{(t)}(\mathcal F)|\ge {x\choose t} \ \ \ \ \ \ \ \ \ \text{holds.}
\end{equation}
\end{thm*}

Note that for $x\ge k-1$ the polynomial ${x\choose k}$ is a monotone increasing function of $x$. Thus $x$ is uniquely determined by $|\mathcal F|$ and $k$.

Hilton observed that the lexicographic order is very useful for handling cross-intersecting families.
\begin{thm*}[Hilton's Lemma \cite{Hil}] If $\mathcal A\subset{[n]\choose a}$ and $\bb\subset{[n]\choose b}$ are cross-intersecting then $\mathcal L^{(a)}(|\mathcal A|)$ and $\mathcal L^{(b)}(|\mathcal B|)$ are cross-intersecting as well.
\end{thm*}
Since it appears that the proof of Hilton's lemma was never published, we give it here for completeness.
\begin{proof}We may clearly assume that $n\ge a+b$. Consider the family $\bar{\mathcal B}$ of complements of sets from $\bb$. We have $|\bar{\mathcal B}| = |\bb|$ and $\bar{\mathcal B}\subset{[n]\choose n-b}$. Since $\aaa$ and $\bb$ are cross-intersecting, then there are no  $\bar B\in \bar{\mathcal B}, A\in \mathcal A$, such that $A\subset \bar B$. In other words, $\mathcal A$ and $\mathcal B$ are cross-intersecting if and only if $\mathcal A$ and $\mathcal S^{(a)}(\bar{\mathcal B})$ are disjoint.

Next, let us define {\it the reversed colex order}. It is the colex order for $[n]$, in which the order of elements is inverted: $n$ is the first element, then $n-1$, etc. Formally, $F\prec G$ in the reversed colex order if $\min F\setminus G>\min G\setminus F.$ It is not difficult to see the following three things. First, if an $a$-set $A$ is the $i$-th in the lexicographical order, then it is $({n\choose a}-i)$-th in the reversed colex order. Second, the complement $[n]-A$ is $({n\choose a}-i)$-th in the lexicographical order and, thus, $i$-th in the reversed colex order. Finally, we note that if a family $\mathcal F\subset{[n]\choose f}$ form an initial segment in the reversed colex order, then for any $0< f'<f$ its shadow $\mathcal S^{(f')}(\mathcal F)$ also forms an initial segment in the reversed colex order.

Returning to the proof of the theorem, we note that $|\mathcal S^{(a)}(\bar{\mathcal B})|\ge |\mathcal S^{(a)}(\bar{\mathcal L_b})|,$ where $\mathcal L_b:=\mathcal L^{(b)}(|\mathcal B|),$ due to (\ref{eq5}) and the considerations in the previous paragraph, which give that $\bar{\mathcal L_b}$ in an initial segment in the reversed colex order. We only point out that we can apply \eqref{eq5} with both colex and reversed colex orders.

Assume that $\mathcal L_a:=\mathcal L^{(a)}(|\mathcal A|)$ and $\mathcal L_b$ are not cross-intersecting. Then $\mathcal L_a$ and $\mathcal S^{(a)}(\bar{\mathcal L_b})$ intersect. The first family consists of the first sets in the lexicographical order, while the second is the initial segment in the reversed colex order, and so consists of the last elements in the lexicographical order (see two paragraphs above). Therefore, we get that $|\mathcal L_a|+|\mathcal S^{(a)}(\bar{\mathcal L_b})|>{n\choose a}.$ But then $|\mathcal A|+|\mathcal S^{(a)}(\bar{\mathcal B})|>{n\choose a}$, which implies that $\mathcal A$ and $\mathcal S^{(a)}(\bar{\mathcal B})$ intersect. This means that $\mathcal A, \mathcal B$ are also not cross-intersecting, a contradiction.
\end{proof}

Looking at the complements the next statement follows.
\begin{cor*} If $\mathcal D\subset{[n]\choose d}$ and $\mathcal E\subset{[n]\choose e}$ are cross-union then $\mathcal C^{(d)}(|\mathcal D|)$ and $\mathcal C^{(e)}(|\mathcal E|)$ are cross-union as well.

\end{cor*}
We use the following standard notation. Given a family $\mathcal A\subset{[n]\choose k}$, the family $\mathcal A(\bar ij)$ is defined in the following way: $\mathcal A(\bar i j) = \{A-\{j\}: A\in \aaa, j\in A, i\notin A\}.$

Let us conclude this section with a simple inequality involving binomial coefficients

\begin{equation}\label{eq10}{n-i\choose k-i}{n\choose k}<{n-i+1\choose k-i+1}{n-1\choose k-1} \ \ \ \ \ \ \ \text{holds for}\ \ n\ge 2k,\ i\ge 2.\end{equation}
\begin{proof}  For $k<i$ the LHS is $0$. Suppose $k\ge i$ and divide both sides by ${n\choose k}{n-i+1\choose k-i+1}$. We obtain $\frac{k-i+1}{n-i+1}<\frac kn$, which is obviously true.
\end{proof}

\section{Proofs}
\subsection{Short proof of Pyber's theorem}
By symmetry, we suppose $|\mathcal A|\le |\mathcal B|$. First note that if $|\mathcal A|\le {n-2\choose k-2}$, then $$|\mathcal A||\mathcal B|\le {n-2\choose k-2}{n\choose k}  <{n-1\choose k-1}^2 \ \ \ \ \ \text{by (\ref{eq10}),\ \  case }i=2.$$

From now on we assume that ${n-2\choose k-2}\le |\mathcal A|\le |\mathcal B|.$ By Hilton's Lemma  we suppose that $\mathcal A = \mathcal L(|\mathcal A|)$, $\mathcal B = \mathcal L(|\mathcal B|),$ i.e., both families are initial segments in the lexicographic order.

Note that the first ${n-2\choose k-2}$ sets in the lexicographic order are all the $k$-sets that contain $1$ and $2$. Since $\mathcal A,\mathcal B$ are cross-intersecting, we infer that  all their members must contain either $1$ or $2$. We shall use this fact to prove:

\begin{prop}\label{prop1} We have \begin{equation}\label{eq22}|\mathcal A|+|\mathcal B|\le 2{n-1\choose k-1}.\end{equation}
\end{prop}
Note that (\ref{eq22}) implies (\ref{eq2}) by the inequality between arithmetic and geometric means. One can even deduce that (\ref{eq2}) is strict unless $|\mathcal A| = |\mathcal B| = {n-1\choose k-1}$ holds.
\begin{proof}[Proof of Proposition \ref{prop1}] If $|\mathcal B|\le {n-1\choose k-1}$ then (\ref{eq22}) is obvious. Therefore, we assume $|\mathcal B|>{n-1\choose k-1}.$ Note that the first ${n-1\choose k-1}$ members of ${[n]\choose k}$ are all the $k$-sets containing $1$. Since $\mathcal A, \mathcal B$ are cross-intersecting, $1\in A$ holds for all $A\in \mathcal A$.
Let $\mathcal B'$  be the family of the remaining sets in $\mathcal B$, i.e., $$\mathcal B' = \{B\in\mathcal B: 1\notin B\}.$$
Let $\mathcal C = \{C\in {[n]\choose k}: 1\in C, C\notin \mathcal A\}.$ To prove (\ref{eq22}) we need to show that $$|\mathcal C|\ge |\mathcal B'|\ \ \ \ \text{holds.}$$

Recall that  \textit{all} $k$-sets containing both $1$ and $2$ are in $\mathcal A$ and therefore all members of $\mathcal B$ contain $1$ or $2$. We infer that $ B\cap \{1,2\} = \{2\}$ for all $B\in \mathcal B'$ and $C\cap \{1,2\} = \{1\}$ for all $C\in \mathcal C$.

Let us now consider a bipartite graph $\mathcal G = (X_1,X_2,E)$, where $X_i := \bigl\{D_i\in {[n]\choose k}: D_i\cap \{1,2\} = \{i\}\bigr\}$ and two vertices $D_1$ and $D_2$ are connected by an edge if and only if $D_1\cap D_2 = \emptyset$ holds.

Note that $\mathcal G$ is regular of degree ${n-k-1\choose k-1}$, $\mathcal C\subseteq X_1,\mathcal B'\subseteq X_2$ hold. Moreover, the cross-intersecting property implies that if $D_1$ and $D_2$ are connected for some $D_2\in \mathcal B'$ then $D_1\in \mathcal C$. In other words, the full neighborhood of $\mathcal B'$ in the regular bipartite graph $\mathcal G$ is contained in $\mathcal C$. This implies $|\mathcal C|\ge |\mathcal B'|$ and concludes the proof.\end{proof}

\subsection{Proof of Theorem \ref{thm1} modulo Theorem \ref{thm2}}
Set $l = n-k $ and consider the families $\mathcal F = \aaa^c, \mathcal G = \bb^c\subset{[n]\choose l}.$ Then $\mathcal F,\mathcal G$ are cross-union. We assume $|\mathcal F|\le |\mathcal G|.$
Note that for $n = 2k = 2l$ the cross-intersecting and cross-union conditions are equivalent and simply mean that if, say, $F\in \mathcal F$ then $[n]-F\notin \mathcal G.$ Therefore, for an arbitrary $\mathcal F\subset {[n]\choose l}$ the families $\mathcal F$ and ${[n]\choose k}-\mathcal F^c$ are cross-union. Moreover, these are altogether ${2k\choose k} = 2{2k-1\choose k-1}$ sets. Consequently, in this case there are many ways to achieve equality in (\ref{eq2}).

The case $|\bb| = {n-1\choose k-1}$ of Theorem \ref{thm1} is somewhat special because replacing $\mathcal B$ by $\mathcal L^{(k)}\bigl({n-1\choose k-1}\bigr)$ would produce the family of all $k$-sets containing $1$, i.e., the family with the intersection of all its members being non-empty. Fortunately, in this case we can apply the theorem of M\"ors.

Setting $\mathcal G = \bb^c,$ by (\ref{eq6}) we have
$$|\mathcal S^{(k)}(\mathcal G)|\ge{n-1\choose k}+{n-k-1\choose k-1}, \ \ \ \ \ \ \text{yielding}$$
$$|\aaa|\le {n\choose k} -|\mathcal S^{(k)}(\mathcal G)|\le {n-1\choose k-1}-{n-k-1\choose k-1},$$
which proves strict inequality in (\ref{eq3}). From now on we may assume $|\bb|>{n-1\choose k-1}$,
and the remaining part of Theorem \ref{thm1} follows from the case $i = k+1$ of Theorem \ref{thm2}, which we prove in the next section.

\subsection{Proof of Theorem \ref{thm2}}
We assume w.l.o.g. that $|\mathcal A|\le |\mathcal B|$. For the whole proof we assume that $\mathcal A, \mathcal B$ are the first sets in the lexicographical order. We consider several cases depending on the size of $|\mathcal A|$.\\

\textbf{Case 1.} $|\mathcal A|\ge {n-2\choose k-2}+{n-i\choose k-i+1}$.

\begin{lem}\label{lem7} If $m\ge 2a$ are natural numbers, $\mathcal A',\mathcal B'\subset{[m]\choose a}$ are cross-intersecting, and for some integer $j\ge 1$ we have ${m-j\choose a-j}\le |\mathcal A'|\le |\mathcal B'|$, then $$|\mathcal A'|+|\mathcal B'|\le {m\choose a}+{m-j\choose a-j}-{m-j\choose a}.$$
\end{lem}
\begin{proof}  We assume that $\mathcal A', \mathcal B'$ are the first sets in the lexicographical order. For $j = 1$ the family$\mathcal A'$ contains all sets containing $\{1\}$. It implies that both families must have cardinality ${n-1\choose a-1}$, since otherwise $\mathcal B'$ contains the set $\{2,\ldots, a+1\}$, which is disjoint with $\{1,a+2,\ldots, 2a\}\in \mathcal A'$. At the same time, the right hand side of the displayed equation above is exactly $2{n-1\choose a-1}$ for $j=1$. Therefore, in what follows we may assume that $j\ge 2$. Since $|\mathcal A'|\le |\mathcal B'|,$ we may w.l.o.g. assume that all sets from $\mathcal A'$ contain $\{1\}$ and $\mathcal B'$ contains the family $\mathcal L$ of all the sets containing $\{1\}$. Since $|\mathcal A'|\ge {m-j\choose a-j},$ then $\mathcal A'$ contains the family $\mathcal A_1$ of all $a$-sets that contain $[1,j]$. Therefore, each set from $\mathcal B'$ must intersect $[1,j]$. We denote by $\mathcal B_0$ the family of all $a$-sets that do not contain $\{1\}$ and intersect $[2,j]$. By $\mathcal A_0$ we denote the family  $\mathcal L\setminus \mathcal A_1$. Note that $|\mathcal A_0| = {m-1\choose a-1}-{m-j\choose a-j}$ and $|\mathcal B_0| = {m-1\choose a} - {m-j\choose a}$.

Consider a bipartite graph $G$ with parts $\mathcal A_0,\mathcal B_0$ and with two sets joined by an edge if they are disjoint. We know that $(\mathcal B'\setminus \mathcal L)\cup(\mathcal A'\setminus \mathcal A_1)$ is an independent set in $G$, since any pair of sets from different families intersect.

We aim to show that there is a matching of $\mathcal A_0$ into $\bb_0$ in $G$. We look at the following decomposition: $\mathcal A_0 = \mathcal P_1\sqcup\ldots\sqcup \mathcal P_{j-1}$, where for any $1\le s\le j-1$ we have
$$\mathcal P_s = \{A\in \mathcal A_0:[2,s]\subset A, s+1\notin A\}.$$ Analogously, we consider the decomposition for $\mathcal B_0:$ $\mathcal B_0 = \mathcal Q_1\sqcup\ldots\sqcup \mathcal Q_{j-1}$, where for any $1\le s\le j-1$ we have
$$\mathcal Q_{s} = \{B\in \mathcal B_0: B\cap [2,s] = \emptyset, s+1\in B\}$$.

We claim that for each $s$ there is a matching of $\mathcal P_s$ into $\mathcal Q_s$. Indeed, $G|_{\mathcal P_s,\mathcal Q_s}$ is a biregular graph with $|\mathcal P_s| = {m-s-1\choose a-s}\le {m-s-1\choose a-1} = |\mathcal Q_s|$, therefore, it has a matching exhausting the smaller part by the K\"onig-Hall theorem. The inequality between two binomial coefficients holds since $m-s-1\ge (a-s)+(a-1) =2a-s-1$.

Since $\mathcal Q_s$ for different $s$ are disjoint,  combining matchings of $\mathcal P_s$ into $\mathcal Q_s$ we get a matching of $\mathcal A_0$ in $\mathcal B_0$. Thus the biggest independent set in $G$ is $\mathcal B_0$ and, therefore, $|\mathcal B'\setminus \mathcal L|+|\mathcal A'\setminus \mathcal A_1|\le |\mathcal B_0| = {m-1\choose a} - {m-j\choose a},$ yielding

$$|\mathcal A'|+|\mathcal B'|\le {m-1\choose a} - {m-j\choose a}+{m-1\choose a-1}+{m-j\choose a-j}= {m\choose a}+{m-j\choose a-j}-{m-j\choose a}. $$
The proof is complete.
\end{proof}

\textbf{Remark.} Actually Lemma \ref{lem7} was proved in a more general form in \cite{FT}. However, the proof that we presented here is shorter and more elementary.\\

To conclude the proof of the theorem in this case one has to notice the following. Recall that $|\mathcal B|\ge {n-1\choose k-1}+{n-i\choose k-i+1}$. Consider the cross-intersecting families $\mathcal B(\bar 12), \mathcal A(1\bar 2)$. We remark that $|\mathcal A|= {n-2\choose k-2}+|\mathcal A(1\bar 2)|$ and $|\mathcal B|= {n-1\choose k-1}+|\mathcal B(\bar 12)|$. Both families are subsets of ${[3,n]\choose k-1}$, moreover, we know that both $|\mathcal B(\bar1 2)|,|\mathcal A(1\bar 2)|\ge {n-i\choose k-i+1}$. Applying Lemma \ref{lem7} with $m = n-2, a= k-1$ and $j = i-2$, we get that

$$|\mathcal B(\bar1 2)|+|\mathcal A(1\bar 2)|\le {n-2\choose k-1}+{n-i\choose k-i+1}-{n-i\choose k-1}.$$
Therefore, $|\mathcal B|+|\mathcal A|\le {n-1\choose k-1}+{n-2\choose k-2}+{n-2\choose k-1}+{n-i\choose k-i+1}-{n-i\choose k-1} = 2{n-1\choose k-1}+{n-i\choose k-i+1}-{n-i\choose k-1}.$
We know that  $|\aaa| \leq  {n-1\choose k-1}-{n-i\choose k-1}$. Knowing the bound on $|\aaa|+|\bb|$ above, it follows that the product $|\aaa||\bb|$ is the biggest if $|\aaa|$ is maximum possible and $|\bb|= \max (|\aaa|+|\bb|) - \max |\aaa|$, which gives exactly  $|\mathcal B| = {n-1\choose k-1}+ {n-i\choose k-i+1}$ and $|\mathcal A| = {n-1\choose k-1}-{n-i\choose k-1}$. The proof in this case is complete.\\

\textbf{Remark.} The argument above and Lemma \ref{lem7} together show that the bound in (\ref{eq4}) actually decreases as $i$ decreases.\\

\textbf{Case 2.} $|\mathcal A|\le {n-3\choose k-3}$. In this case
$$|\mathcal A||\mathcal B|\le {n\choose k}{n-3\choose k-3}<{n-1\choose k-1}{n-2\choose k-2}<{n-1\choose k-1}\Bigl({n-1\choose k-1}-{n-i\choose k-1}\Bigr).$$
The right hand side is obviously less than the right hand side of (\ref{eq4}). The second inequality follows from (\ref{eq10}) with $i = 3$. The proof of (\ref{eq4}) in this case is complete.\\

\textbf{Case 3.} ${n-3\choose k-3}\le |\mathcal A|\le {n-2\choose k-2}$. For this case we are going to pass to the complements of sets from $\mathcal A,\mathcal B$ and to change from cross-intersecting to cross-union families. Set $l = n-k$. Note that $2l \ge n$. We may assume that both $\mathcal A,\mathcal B$ consist of the initial segments of $l$-sets in the colex order.

First we verify that when $|\mathcal A|={n-2\choose k-2} = {n-2\choose l}$, the inequality  (\ref{eq4}) holds. Indeed, then $|\mathcal B| \le {n-1\choose l}+{n-2\choose l-1}$. At the same time, the right hand side of (\ref{eq4}) is the smallest when $i = 3$ and then it is strictly bigger than $f= {n-1\choose k-1}\Bigl({n-2\choose k-2}+{n-3\choose k-2}\Bigr) = {n-1\choose l}\Bigl({n-2\choose l}+{n-3\choose l-1}\Bigr).$ Therefore, in this case

\begin{small}\begin{equation}\label{eq16}f - |\mathcal A||\mathcal B| \ge {n-1\choose l}{n-3\choose l-1}-{n-2\choose l}{n-2\choose l-1} = \Bigl(\frac{(n-1)l}{l(n-2)}-1\Bigr){n-2\choose l}{n-2\choose l-1}>0.\end{equation}
\end{small}

Next we pass to the case when $|\mathcal B| = {n-1\choose l}+{n-2\choose l-1}+{x\choose l-2}$ for some $x\ge l-2$. We remark that $x\le n-3$ since $|\mathcal A|\ge {n-3\choose l}$. In this case, by (\ref{eq7}), we have $$|\mathcal A|\le {n-2\choose l}-{x\choose n-l-2}.$$ Indeed, $\mathcal A$ contains sets from ${[n-2]\choose l}$ only, while $\mathcal B$ contains all sets from ${[n-1]\choose l}$, all sets from $\big\{\{n\}\cup F: F\in {[n-2]\choose l-1}\big\},$ as well as ${x\choose l-2}$ sets from $\big\{\{n-1,n\}\cup F: F\in{[n-2]\choose l-2}\big\}$. Let us denote the family of sets $\{B\setminus \{n-1,n\}:\{n-1,n\}\subset B\in \bb\}$ by $\mathcal B(n-1,n)$. We have $|\mathcal B(n-1,n)|={x\choose l-2}$. We think of this family and the family $\mathcal A$ as families of subsets of $[n-2]$. These families are cross-union in $[n-2]$. Therefore, $\mathcal A\cap\big\{[n-2]-S: S\in \mathcal S^{(n-l-2)}(\bb(n-1,n)) \big\}=\emptyset$, which, together with \eqref{eq7}, implies the displayed inequality above.

 To verify (\ref{eq4}) in this case it is sufficient for us to prove that
\begin{equation}\label{eq17} {x\choose l-2}{n-2\choose l}\le {x\choose n-l-2}{n-1\choose l}.\end{equation}
If we do so, then the value of the product for any $n-3\ge x\ge l-2$ is not bigger than the value of the product, calculated for the case $|\mathcal A|={n-2\choose l}$, which, in turn, is smaller than the right hand side of (\ref{eq4}).

It is easy to see that, since $l-2\ge n-l-2$, the function ${x\choose l-2}/{x\choose n-l-2}$ is a monotone increasing function. Therefore, it is sufficient to verify (\ref{eq17}) for $x = n-3$. In that case (\ref{eq17}) transforms into ${n-3\choose l-2}{n-2\choose l}\le {n-3\choose l-1}{n-1\choose l},$ which is true since $\frac{l-1}{n-2}\cdot\frac l{n-1}<\frac l{n-2}.$\\

\textbf{Case 4.} ${n-2\choose k-2}\le |\mathcal A|\le {n-2\choose k-2}+{n-i\choose k-i+1}$. Before starting the proof in this case we remark that for $i = k+1$ this case is not necessary, since it is covered by cases 1 and 3. Since for $i = 3$ the right hand side of (\ref{eq4}) is the smallest and the interval of values for $|\mathcal A|$ is the largest, we may w.l.o.g. assume that $i = 3$.

We remark that for $|\mathcal A| = {n-2\choose l}+{n-3\choose l-1}$ we have $|\mathcal B| = {n-1\choose l}+{n-3\choose l-1}$ and we obtain exactly the bound (\ref{eq4}). Now assume that
$|\mathcal B| = {n-1\choose l}+{n-3\choose l-1}+{x\choose l-2}.$ Then, again using (\ref{eq7}), we get that $|\mathcal A|\le {n-2\choose l}+{n-3\choose l-1}-{x\choose n-l-2}$. If $x\le n-4$, then ${x\choose n-l-2}\ge {x\choose l-2}$ and, therefore, the product $|\mathcal A||\mathcal B|$ is smaller than the right hand side of (\ref{eq4}).


The last remaining case is that $|\mathcal A|\le {n-2\choose l}+{n-3\choose l-1} - {n-4\choose n-l-2} = {n-2\choose l}+{n-4\choose l-1}$. In this case we have $|\mathcal A||\mathcal B|\le \Bigl({n-2\choose l}+{n-4\choose l-1}\Bigr)\Bigl({n-1\choose l}+{n-2\choose l-1}\Bigr).$ We claim that this value is less than the value of the right hand side of (\ref{eq4}) for $i = 3$:
$$\Bigl({n-2\choose l}+{n-4\choose l-1}\Bigr)\Bigl({n-1\choose l}+{n-2\choose l-1}\Bigr)\le \Bigl({n-1\choose l}+{n-3\choose l-1}\Bigr)\Bigl({n-2\choose l}+{n-3\choose l-1}\Bigr).$$

Indeed, we just check for any $j_1,j_2 \in \{1,2\}$ that the product of the $j_1$-th summand from the first bracket and the $j_2$-th summand from the second bracket in the left hand side is not bigger than the corresponding product in the right hand side. For $j_1 = j_2 = 1$ it is obvious and for $j_1 = 1, j_2 = 2$ it is shown in (\ref{eq16}). As for the rest, we have

$${n-4\choose l-1}{n-1\choose l}\le \frac{(n-3)(n-l-1)}{(n-l-2)(n-1)}{n-4\choose l-1}{n-1\choose l}= {n-3\choose l-1}{n-2\choose l}$$
since $((n-1)-2)(n-l-1)-(n-1)((n-l-1)-1) = n-1-2(n-l-1) = 2l-n+1>0,$ and
$${n-4\choose l-1}{n-2\choose l-1}\le \frac{(n-3)(n-l-1)}{(n-l-2)(n-2)}{n-4\choose l-1}{n-2\choose l-1}= {n-3\choose l-1}{n-3\choose l-1}$$
since $((n-2)-1)((n-l-2)+1)-(n-2)(n-l-2) = n-2-(n-l-2)-1 = l-1>0.$ The last inequality is due to the well-known fact that  ${m \choose b}$  is a $\log$-concave function of $m$, but to make the argument self-contained, we included the proof. The proof of the theorem is complete.

\section{Acknowledgements} We want to thank the anonymous reviewer for carefully reading the manuscript and pointing out several unclear places in the proof.

\end{document}